\documentclass[12pt, english]{article}

\usepackage{lipsum}
\usepackage{etoolbox}
\RequirePackage{ifthen}

\usepackage{amsmath,amssymb,amsfonts}
\usepackage{mathrsfs}
\usepackage{enumitem}

\input{xy}
\xyoption{all}
\RequirePackage[top=4cm,bottom=3.5cm,inner=3.5cm,outer=3.5cm,
headsep=1cm,footskip=1cm]{geometry}

\RequirePackage{xr-hyper}
\RequirePackage{hyperref}

\hypersetup{bookmarksopen=true}%

\ifthenelse{\boolean{pdf}}{%
	\hypersetup{colorlinks=true}}{%
	\hypersetup{colorlinks=false}}

\hypersetup{linkcolor=blue}%
\hypersetup{urlcolor=cyan}%
\hypersetup{citecolor=cyan}

\makeatletter
\let\@afterindenttrue\@afterindentfalse
\patchcmd\@thm
{\let\thm@indent\indent}{\let\thm@indent\noindent}%
{}{}
\makeatother

\newtheorem{thm}{Theorem}[section]
\newtheorem{lemma}[thm]{Lemma}

\newtheorem{myexamp}{Example}[section]
\newtheorem{definition}[myexamp]{Definition}
\newtheorem{remark}[thm]{Remark}

\newcommand{\qed}{\hfill$\square$}

\newcommand{\R}{\mathbb{R}}
\newcommand{\N}{\mathbb{N}}
\newcommand{\K}{\mathbb{K}}
\newcommand{\ideal}[1]{\ensuremath{\mathfrak{#1}}}

\begin{document}
	
\begin{center}
	\Large{\textbf{Normality of Vaserstein group}}
\end{center}
	
\begin{center}
	Ruddarraju Amrutha\footnote{ This article is part of the doctoral thesis of the first named author}, Pratyusha Chattopadhyay
\end{center}
	
	
\medskip


\begin{center}
	Abstract
\end{center}	
	
A.A. Suslin proved a normality theorem for an elementary linear group and V.I. Kopeiko extended this result of Suslin for a symplectic group defined with respect to the standard skew-symmetric matrix of even size. We generalized the result of Kopeiko for a symplectic group defined with respect to any invertible skew-symmetric matrix of Pfaffian one. Vaserstein group is an extension of a symplectic group defined with respect to any invertible skew-symmetric matrix of Pfaffian one in the set up of projective modules. Here we prove a normality theorem for Vaserstein group.

\medskip


\section{Introduction}
\label{section:1}

In \cite{Sus}, A.A. Suslin proved a normality theorem for the elementary linear group, which says that for a commutative ring $R$ with $1$, the elementary linear group $\mathrm{E}_n(R)$ is normal in the general linear group $\mathrm{GL}_n(R)$, for $n\geq 3$. This normality theorem is important as it was used in proving a $K_1$-analogue of Serre's conjecture. Serre's conjecture says that for a field $\K$, any finitely generated projective module over the polynomial ring $\K[x_1,\ldots,x_r]$ is free. The $K_1$-analogue of this result, proved by Suslin in \cite{Sus}, says that every element in the special linear group $\mathrm{SL}_n(\K[x_1,\ldots,x_r])$ can be written as a product of elementary matrices over $\K[x_1,\ldots,x_r]$, for $n\geq 3$. Suslin also proved a stronger relative version of the above mentioned normality theorem with respect to an ideal.
	
A similar normality theorem was proved by V.I. Kopeiko in \cite{Kop} for the elementary symplectic group. In the symplectic case, the relative version of the normality theorem says that for a commutative ring $R$ with $R=2R$ and an ideal $I$ of $R$, the relative elementary symplectic group $\mathrm{ESp}_{2n}(R,I)$ is a normal subgroup of $\mathrm{Sp}_{2n}(R)$, for $n\geq 2$. Kopeiko used this result to prove that for a field $\K$, any element of $\mathrm{Sp}_{2n}(\K[x_1,\ldots,x_r])$ can be written as a product of elementary symplectic matrices over $\K[x_1,\ldots,x_r]$, for $n\geq 2$. In \cite{Kop-Sus}, Suslin and Kopeiko together proved a normality theorem for orthogonal groups which they used to prove that if $\K$ is a field such that char$(\K) \ne 2$, and $A=\K[x_1,\ldots,x_r]$, then any quadratic $A$-space is extended from $\K$.
	
In \cite{Norm}, authors generalized Kopeiko's normality theorem. The main result (Theorem 5.2 of \cite{Norm}) says that if $\varphi$ is an invertible skew-symmetric matrix of Pfaffian $1$ of size $2n$ over $R$ with $R=2R$, and $I$ is an ideal of $R$, then the relative elementary symplectic group $\mathrm{ESp}_\varphi(R,I)$ with respect to the matrix $\varphi$ is normal in $\mathrm{Sp}_\varphi(R)$, the symplectic group with respect to $\varphi$.   

In this paper we deal with a group which is considered as generalization of $\mathrm{Sp}_\varphi(R)$ in the set up of projective modules. In 1955 J.-P. Serre stated his famous conjecture, and about twenty years later, in the year of 1976, D. Quillen and A.A. Suslin proved this conjecture independently (see \cite{Q}, \cite{Su}). However, during this gap of twenty years many partial results on this conjecture were established which are known as ``classical" results on Serre's conjecture, and in the latter half H. Bass introduced the concept of transvections to obtain cancellation theorems involving projective modules which in turn was used to prove some of the ``classical" results (see \cite{B2}, \cite{Sw}). Thus transvection groups are important class of subgroups of the group of automorphisms of projective modules.

It is well-known that for a free module $M=R^n$ a linear transvection looks like $I_n + v w^t$, where $v, w \in R^n$, one of $v, w$ is unimodular, and their dot product $v^t w = 0$. Hence it follows that in the case of free modules linear transvection groups coincide with the elementary linear groups. Using similar kind of arguments it can be deduced that in the case of free modules symplectic transvection groups  coincide with the elementary symplectic groups and orthogonal transvection groups coincide with the elementary orthogonal groups (see \cite{Sus}, \cite{Kop}, \cite{Kop-Sus} for proofs in linear, symplectic, and orthogonal cases respectively). Bak, Basu, Rao generalized the normality theorems of Suslin and Kopeiko for the linear and symplectic transvection groups. 

In this paper we recall definitions of Vaserstein group in the absolute case and in the relative case (with respect to an ideal of a ring) as they are defined in Definition \ref{V(P) defn} and Definition \ref{V(P,IP) defn}. Note that Vaserstein groups can be considered as generalization of $\mathrm{ESp}_\varphi(R)$ in the set up of projective modules (see Lemma \ref{lemma:V=ESp}). Here we prove a normality result for Vaserstein group (see Lemma \ref{lemma:normality}). This result can considered as an extension of normality result for $\mathrm{ESp}_\varphi(R)$ in the case of projective modules.


\section{Preliminaries}
\label{section:2}
	
Let $R$ denote a commutative ring with unity. Let $R^n$ denote the space of column vectors of length $n$ with entries in $R$. The ring of matrices of size $n\times n$ with entries in $R$ is denoted by $\mathrm{M}_{n}(R)$. The identity matrix of size $n\times n$ is denoted by $I_n$ and $e_{ij}$ denotes the $n\times n$ matrix which has $1$ in the $(i,j)$-th position and $0$ everywhere else. The collection of invertible $n\times n$ matrices with entries in $R$ is denoted by $\mathrm{GL}_{n}(R)$. For $\alpha\in\mathrm{M}_{m}(R)$ and $\beta\in\mathrm{M}_{n}(R)$, the matrix {\tiny{$\begin{pmatrix}\alpha&0\\0&\beta\end{pmatrix}$}}, which is an element of $\mathrm{M}_{m+n}(R)$, is denoted by $\alpha\perp\beta$. For a matrix $\alpha$ of size $m\times n$, its transpose is denoted by $\alpha^t$, a matrix of size $n\times m$. For an $R$-module $M$, we denote by $M^\ast$ the dual space of $M$, that is, the collection of all $R$-module homomorphisms $f:M\rightarrow R$. 
	
\begin{definition} 
	\rm{The \textit{elementary linear group} $\mathrm{E}_{n}(R)$ is a subgroup of $\mathrm{GL}_n(R)$ generated by elements of the form $E_{ij}(a)=I_n+ae_{ij}$, for $a\in R$. For an ideal $I$ of $R$, the subgroup of $\mathrm{E}_{n}(R)$ generated by $E_{ij}(x)$, for $x\in I$ is denoted by $\mathrm{E}_{n}(I)$. The \textit{relative elementary group}, denoted by $\mathrm{E}_{n}(R,I)$, is the normal closure of $\mathrm{E}_{n}(I)$ in $\mathrm{E}_{n}(R)$. In other words, $\mathrm{E}_{n}(R,I)$ is generated by elements of the form $E_{kl}(a) E_{ij}(x) E_{kl}(-a)$, where $a \in R$ and $x \in I$.}
\end{definition} 
	
\begin{definition} 
	\rm{Let $\psi_n=\sum_{i=1}^{n}(e_{2i-1,2i}-e_{2i,2i-1})$ denote the standard skew-symmetric matrix. The \textit{symplectic group}, denoted by $\mathrm{Sp}_{2n}(R)$, is a subgroup of $\mathrm{GL}_{2n}(R)$ defined as 
		\begin{equation*}
			\mathrm{Sp}_{2n}(R)=\{\alpha\in\mathrm{GL}_{2n}(R)\;\big|\;\alpha^t\psi_n\alpha=\psi_n\}.
		\end{equation*}	
		Given an invertible skew-symmetric matrix $\varphi$ of size $2n$, 
	\begin{equation*}
		\mathrm{Sp}_\varphi(R)=\{\alpha\in\mathrm{GL}_{2n}(R)\;\big|\;\alpha^t\varphi\alpha=\varphi\}.
	\end{equation*}}	
\end{definition} 
	
\begin{definition} 
	\rm{Let $I$ be an ideal of $R$. Then we have the canonical ring homomorphism $f:R\rightarrow R/I$. Using $f$, we can define a ring homomorphism $\tilde{f}:\mathrm{Sp}_\varphi(R)\rightarrow\mathrm{Sp}_{\varphi}(R/I)$ given by $\tilde{f}(a_{ij})=(f(a_{ij}))$. We denote the kernel of this map by $\mathrm{Sp}_{\varphi}(R,I)$.}
\end{definition} 
	
\begin{definition} 
	\rm{Let $\sigma$ be the permutation of $\{1,2,\ldots,2n\}$ given by $\sigma(2i-1)=2i$ and $\sigma(2i)=2i-1$. For $a\in R$ and $1\leq i\neq j\leq 2n$, define $se_{ij}(a)$ as
	\begin{equation*}
		se_{ij}(a)=\begin{cases} I_{2n}+ae_{ij},&\text{ if } i=\sigma(j)\\
					I_{2n}+ae_{ij}-(-1)^{i+j}ae_{\sigma(j)\sigma(i)},&\text{ if } i\neq\sigma(j)\\ \end{cases}.
	\end{equation*}
	Note that $se_{ij}(a)\in\mathrm{Sp}_{2n}(R)$. These matrices are called the \textit{elementary symplectic matrices}. The subgroup of $\mathrm{Sp}_{2n}(R)$ generated by the elementary symplectic matrices is called the \textit{elementary symplectic group}, and is denoted by $\mathrm{ESp}_{2n}(R)$. We denote by $\mathrm{ESp}_{2n}(I)$ a subgroup of $\mathrm{ESp}_{2n}(R)$ generated by elements of the form $se_{ij}(x)$, for $x\in I$. The \textit{relative elementary group}, denoted by $\mathrm{ESp}_{2n}(R,I)$, is the normal closure of $\mathrm{ESp}_{2n}(I)$ in $\mathrm{ESp}_{2n}(R)$. In other words, $\mathrm{ESp}_{2n}(R, I)$ is generated by elements of the form $se_{kl}(a) se_{ij}(x) se_{kl}(-a)$, where $a \in R$ and $x \in I$.}
\end{definition} 
	
\begin{definition} 
	\rm{Let $\varphi$ be an invertible skew-symmetric matrix of size $2n$ of the form $\begin{pmatrix} 0 & -c^t\\ c & \nu \end{pmatrix}$, and $\varphi^{-1}$ be of the form $\begin{pmatrix} 0 & d^t\\ -d & \mu \end{pmatrix}$, where $c,d\in R^{2n-1}$ and $\nu,\mu\in \textrm{M}_{2n-1}(R)$.
		
	Given $v\in R^{2n-1}$, consider the matrices $\alpha$ and $\beta$ defined as 
	\begin{equation*}
		\begin{aligned}
		\alpha&:=\alpha_\varphi(v)&:=I_{2n-1}+dv^t\nu\\
		\beta&:=\beta_\varphi(v)&:=I_{2n-1}+\mu vc^t.
		\end{aligned}
	\end{equation*}
		
	L.N. Vaserstein constructed these matrices in Lemma 5.4, \cite{Vas}. Note that $\alpha$ and $\beta$ depend on $\varphi$ and $v$. Also, $\alpha,\beta\in\mathrm{E}_{2n-1}(R)$. This follows by Corollary 1.2 and Lemma 1.3 of \cite{Sus}. One can also see Lemma 9.11 of Chapter 1 of \cite{Lam} for a proof of this result. An interesting observation about these matrices is that $\mathrm{E}_{2n-1}(R)$ is generated by the set  $\{\alpha_\varphi(v),\beta_\varphi(v)\; :\; v\in R^{2n-1}\}$ (Theorem 5.1 of \cite{GSV}). Using these matrices, Vaserstein constructed the following matrices in \cite{Vas}:
	\begin{equation*}
		\begin{aligned}
		C_\varphi(v)&:=\begin{pmatrix} 1&0\\ v&\alpha\end{pmatrix}\\
		R_\varphi(v)&:=\begin{pmatrix} 1&v^t\\ 0&\beta\end{pmatrix}.
		\end{aligned}
	\end{equation*}
		
	Note that $C_\varphi(v)$ and $R_\varphi(v)$ belong to $\mathrm{Sp}_\varphi(R)$. The \textit{elementary symplectic group} $\mathrm{ESp}_{\varphi}(R)$ \textit{with respect to the invertible skew-symmetric matrix} $\varphi$ is a subgroup of $\mathrm{Sp}_\varphi(R)$ generated by $C_\varphi(v)$ and $R_\varphi(v)$, for $v\in R^{2n-1}$. We denote by $\mathrm{ESp}_{\varphi}(I)$ a subgroup of $\mathrm{ESp}_{\varphi}(R)$ generated as a group by the elements $C_\varphi(v)$ and $R_\varphi(v)$, for $v\in I^{2n-1}(\subseteq R^{2n-1})$. The \textit{relative elementary symplectic group} $\mathrm{ESp}_{\varphi}(R,I)$ is the normal closure of $\mathrm{ESp}_{\varphi}(I)$ in $\mathrm{ESp}_{\varphi}(R)$.}
\end{definition} 
	
\begin{definition} 
	\rm{Let $R$ be a ring and $M$ be an $R$-module. Let $B:M\times M\rightarrow R$ be a bilinear form on $M$. $B$ is said to be \textit{skew-symmetric} if $B(n,m)=-B(m,n)$ for all $m, n\in M$. $B$ is said to be \textit{alternating} if $B(m,m)=0$, for every $m\in M$. $B$ is said to be \textit{nondegenerate} if $M\cong M^\ast$ by $x\mapsto B(x,-)$.}
\end{definition} 

\begin{remark}
	Let $R$ be a ring and $M$ be an $R$-module. Let $B$ be a bilinear form on $R$. If $2$ is a unit in $R$, then $B$ is skew-symmetric if and only if $B$ is alternating.
\end{remark}
	
\begin{definition} 
	\label{projective defn}
	\rm{An $R$-module $P$ is called \textit{projective} if $\mathrm{Hom}_R(P,.)$ is an exact functor from $R$-modules to abelian groups. An $R$-module $P$ is projective if and only if $P$ is a direct summand of a free $R$-module.}
\end{definition} 

\begin{definition} 
	\rm{A matrix $\varphi\in\mathrm{M}_n(R)$ is said to be \textit{alternating} if there exists $\nu\in\mathrm{M}_n(R)$ such that $\varphi=\nu-\nu^t$.}
\end{definition} 
	
\begin{remark}
	Let $P=\oplus_{i=1}^{2n} Re_i$ be a free $R$-module of even rank. The nondegenerate alternating bilinear form $\langle,\rangle$ on $P$ corresponds to an alternating matrix $\varphi$ with Pfaffian $1$ with respect to the basis $\{e_1,e_2,\ldots,e_{2n}\}$ of $P$ and we write $\langle p,q\rangle_\varphi=p^t\varphi q$.
\end{remark}
	
\begin{definition} 
	\rm{A \textit{symplectic $R$-module} is a pair $(P,\langle,\rangle)$, where $P$ is a finitely generated projective module of even rank and $\langle,\rangle:P\times P\rightarrow R$ is a nondegenerate alternating bilinear form.}
\end{definition} 
	
\begin{definition} 
	\rm{An \textit{isometry} $\varphi$ of a symplectic module $(P,\langle,\rangle)$ is an automorphism of $P$ that fixes the bilinear form $\langle,\rangle$, that is, $\langle\varphi(x),\varphi(y)\rangle=\langle x,y\rangle$ for all $x$, $y\in P$. The group of isometries of $(P,\langle,\rangle)$ is denoted by $\mathrm{Sp}(P,\langle,\rangle)$.} 
\end{definition} 


\section{Results about elementary symplectic group}
\label{section:3}
	
In this section, we recall a few results related to the elementary symplectic group including the normality result for elementary symplectic group (\cite{Norm}, Theorem 5.2) that will be used in proving the main result (Lemma \ref{lemma:normality}).

\begin{lemma}(Corollary 1.11, \cite{Kop})
\label{lemma:3.1}
	Let $R$ be a ring and $I$ be an ideal of $R$. Let $n\geq 2$. Then, $\mathrm{ESp}_{2n}(R,I)$ is a normal subgroup of $\mathrm{Sp}_{2n}(R)$.
\end{lemma}

\begin{lemma}
\label{lemma:3.2}
	Let $(R,\ideal{m})$ be a local ring and $\varphi$ be a skew-symmetric matrix of Pfaffian $1$ of size $2n$ over $R$. Then $\varphi=\epsilon^t\psi_n\epsilon$, for some $\epsilon\in\mathrm{E}_{2n}(R)$. 
\end{lemma}

We recollect an observation of Rao-Swan stated in the introduction of \cite{RaoSwan}.

\begin{lemma}
\label{lemma:3.3}
	Let $n \ge 2$ and $\epsilon \in \mathrm{E}_{2n}(R)$. Then there exists $\rho \in \mathrm{E}_{2n-1}(R)$ such that $(1 \perp \rho) \epsilon \in \mathrm{ESp}_{2n}(R)$.
\end{lemma}
	
One can see Lemma 4.4 of \cite{JOA} for a detailed proof. Using the above observation one can prove the following equality.	
	
\begin{lemma}(Corollary 4.5, \cite{JOA})
\label{lemma:3.4}
	For $n\geq 2$ and $\epsilon\in\mathrm{E}_{2n}(R)$, we have an $\epsilon_0\in\mathrm{E}_{2n-1}(R)$ such that $\epsilon^t\psi_n\epsilon= (1\perp\epsilon_0)^t\psi_n(1\perp\epsilon_0)$.
\end{lemma}

\begin{lemma}(Lemma 3.6, 3.7, \cite{JOA})
\label{lemma:3.5}
	Let $\varphi$ and $\varphi^\ast$ be two invertible skew-symmetric matrices such that $\varphi=(1\perp\epsilon)^t\varphi^\ast(1\perp\epsilon)$ for some $\epsilon\in\mathrm{E}_{2n-1}(R)$. Then, we have
	\begin{equation*}
		\begin{aligned}
		\mathrm{Sp}_\varphi(R)&=(1\perp\epsilon)^{-1}\mathrm{Sp}_{\varphi^\ast}(R)(1\perp\epsilon),\\
		\mathrm{ESp}_{\varphi}(R)&=(1\perp\epsilon)^{-1}\mathrm{ESp}_{\varphi^\ast}(R)(1\perp\epsilon).
		\end{aligned}
	\end{equation*}
\end{lemma}

\begin{lemma}(Lemma 3.8, \cite{JOA})
\label{lemma:3.6}
	Let $\varphi$ and $\varphi^\ast$ be two invertible skew-symmetric matrices such that $\varphi=(1\perp\epsilon)^t\varphi^\ast(1\perp\epsilon)$, for some $\epsilon\in\mathrm{E}_{2n-1}(R,I)$. Then,
	\begin{equation*}
		\mathrm{ESp}_{\varphi}(R,I)=(1\perp\epsilon)^{-1}\mathrm{ESp}_{\varphi^\ast}(R,I)(1\perp\epsilon).
	\end{equation*}
\end{lemma}

The following lemma shows that the elementary symplectic group $\mathrm{ESp}_{\varphi}(R)$ with respect to a skew-symmetric matrix $\varphi$ can be considered as a generalization of the elementary symplectic group $\mathrm{ESp}_{2n}(R)$. 

\begin{lemma}(Lemma 3.5, \cite{JOA})
	\label{lemma:3.7}
	Let $R$ be a ring with $R=2R$ and $n\geq 2$. Then,  $\mathrm{ESp}_{\psi_n}(R)=\mathrm{ESp}_{2n}(R)$.
\end{lemma}

%
%

\medskip
Now, we will state a result which gives a relation between the relative elementary symplectic group $\mathrm{ESp}_{2n}(R[X],(X))$ and the group $\mathrm{Sp}_{2n}(R[X],(X))$. The proof of this lemma uses ideas from Lemma 2.7 of \cite{Kop}. 

\begin{lemma}(Lemma 3.5, \cite{AC2})
	\label{ESp cap Sp X}
	For a ring $R$, we have
	\begin{equation*}
		\mathrm{ESp}_{2n}(R[X],(X))=\mathrm{ESp}_{2n}(R[X])\cap \mathrm{Sp}_{2n}(R[X],(X)).
	\end{equation*}
\end{lemma}

\medskip
\noindent
\textbf{Notation:} Let $\varphi$ be an invertible skew-symmetric matrix of size $2n$ over $R$. 
\begin{equation*}
	\mathrm{Sp}_{\varphi\otimes R[X]}(R[X]):=\{\alpha\in\mathrm{GL}_{2n}(R[X])\;\big|\; \alpha^t\varphi\alpha=\varphi\}.
\end{equation*}

By $\mathrm{ESp}_{\varphi\otimes R[X]}(R[X])$, we mean the elementary symplectic group generated by $C_\varphi(v)$ and $R_\varphi(v)$, where $v\in R[X]^{2n-1}$.\\

\medskip
Now, we will state a lemma which is a generalization of Lemma \ref{ESp cap Sp X} with respect to an invertible skew-symmetric matrix.

\begin{lemma}({Lemma 3.14, \cite{Norm}})
\label{lemma:resp=esp cap rsp}
	Let $R$ be a ring with $R=2R$ and $\varphi$ be an invertible skew-symmetric matrix of size $2n$. Then,
	\begin{equation*}
	\mathrm{ESp}_{\varphi\otimes R[X]}(R[X],(X))=\mathrm{ESp}_{\varphi\otimes R[X]}(R[X])\cap\mathrm{Sp}_{\varphi\otimes R[X]}(R[X],(X)).
	\end{equation*} 
\end{lemma}

\medskip
\noindent
\textbf{Notation:} Let $R$ and $S$ be two commutative rings with unity and let $f:R\rightarrow S$ be a ring homomorphism. If $\delta=(a_{ij})$, then define $f(\delta)=(f(a_{ij}))$. For a maximal ideal $\ideal{m}$ of $R$, the group $\mathrm{ESp}_{\varphi\otimes R_\ideal{m}[X]}(R_\ideal{m}[X])$ is the elementary symplectic group with respect to the matrix $f(\varphi)$, where $f:R\rightarrow R_\ideal{m}$ is the map $a\mapsto \frac{a}{1}$.

\begin{lemma}(Lemma 3.10, \cite{Norm})
\label{lemma:3.11}
	Let $\varphi$ be an invertible skew-symmetric matrix of Pfaffian 1 of size $2n$ over $R$ with $R=2R$. Let $\ideal{m}$ be a maximal ideal of $R$. Then, $\mathrm{ESp}_{\varphi\otimes R_\ideal{m}[X]}(R_\ideal{m}[X])$ is a normal subgroup of $\mathrm{Sp}_{\varphi\otimes R_\ideal{m}[X]}(R_\ideal{m}[X])$. 
\end{lemma}

\medskip
Now, we recall an analogue of Lemma \ref{lemma:3.1} for $\mathrm{ESp}_\varphi(R,I)$.
	
\begin{thm}(Theorem 5.2, \cite{Norm})
\label{lemma:norm of esp}
	Let $\varphi$ be a skew-symmetric matrix of Pfaffian $1$ of size $2n$ over $R$ with $n\geq 2$. Assume that $R=2R$. Let $I$ be an ideal of $R$. Then, $\mathrm{ESp}_{\varphi}(R,I)$ is a normal subgroup of $\mathrm{Sp}_\varphi(R)$.
\end{thm}


\section{$P[X]$ as a symplectic module}
\label{section:4}	
	
In this section, we will prove that when the projective module is free, the group of isometries equals the symplectic group. We will also prove that $P[X]$ is projective $R[X]$-module if $P$ is a projective $R$-module (see Lemma \ref{lemma:P[X] is proj}).  Finally, we will define a bilinear from $\langle,\rangle_{\otimes R[X]}$ on $P[X]$ (see Lemma \ref{lemma:P[X] is symplectic}) such that $(P[X],\langle,\rangle_{\otimes R[X]})$ is a symplectic module.

\begin{lemma}
\label{lemma:iso gp=sym gp}
	Let $(P,\langle,\rangle)$ be a symplectic $R$-module where $P$ is free of rank $2n$. Suppose $\langle,\rangle$ corresponds to an invertible alternating matrix $\varphi$ of size $2n$. Then,
	\begin{equation*}
		\mathrm{Sp}(P,\langle,\rangle_\varphi)=\mathrm{Sp}_\varphi(R).
	\end{equation*}
\end{lemma}
\begin{proof}
	As $P$ is free of rank $2n$, we can consider any element $x$ of $P$ as an element of $R^{2n}$ and any homomorphism $\alpha:P\rightarrow P$ as an element of $\mathrm{M}_{2n}(R)$. In such case $\alpha(x)$ can be represented by $\alpha x$ (matrix multiplication). Note that
	\begin{eqnarray*}
		&&\langle\alpha(p),\alpha(q)\rangle_\varphi=\langle p,q\rangle_\varphi \text{ for all } p,q\in P\\
		&\iff&(\alpha p)^t\varphi(\alpha q)=p^t\varphi q \text{ for all } p,q\in P\\
		&\iff&p^t(\alpha^t\varphi\alpha) q=p^t\varphi q \text{ for all } p,q\in P\\
		&\iff&\alpha^t\varphi\alpha=\varphi.
	\end{eqnarray*}
	Therefore $\alpha\in\mathrm{Sp}(P,\langle,\rangle_\varphi)$ if and only if $\alpha\in\mathrm{Sp}_\varphi(R)$ and hence $\mathrm{Sp}(P,\langle,\rangle_\varphi)=\mathrm{Sp}_\varphi(R)$.
\qed
\end{proof}
	
\begin{lemma}
\label{lemma:P[X] is fin gen}
	If $P$ is a finitely generated $R$-module, then $P[X]$ is finitely generated as an $R[X]$ module. 
\end{lemma}
\begin{proof}
	Suppose $\{p_1,p_2,\ldots,p_n\}$ is a finite generating set for $P$ as an $R$-module. Let $q(X)=\sum_i q_iX^i$ be an element of $P[X]$. As $P$ is generated by $\{p_1,p_2,\ldots,p_n\}$, for each $i$, we have $q_i=\sum_j a_{ij}p_j$ for some $a_{ij}\in R$. Then, $q(X)=\sum_j a_jp_j$, where $a_j=\sum_i a_{ij}X^i\in R[X]$ for $j=1,2,\ldots,n$. Therefore, $P[X]$ is generated by $\{p_1,p_2,\ldots,p_n\}$ as an $R[X]$-module and hence $P[X]$ is a finitely generated $R[X]$ module.
\qed 
\end{proof}
	
\begin{lemma}
\label{lemma:P[X] is proj}
	If $P$ is $R$-projective, then $P[X]$ is $R[X]$-projective.    	
\end{lemma}
\begin{proof}
	Suppose $P$ is a projective $R$-module. Then, by the equivalent condition in Definition \ref{projective defn}, there exist $R$-modules $\mathcal{F}$ and $Q$ such that $\mathcal{F}$ is free and $\mathcal{F}=P\oplus Q$. Then, $\mathcal{F}[X]= P[X]\oplus Q[X]$ and $\mathcal{F}[X]$ is a free $R[X]$-module. So, $P[X]$ is a direct summand of a free $R[X]$-module and hence is $R[X]$-projective by the equivalent condition in Definition \ref{projective defn}.
\qed
\end{proof}
	
\begin{remark}
\label{remark:surjetive}
	Suppose $\psi: P[X]\rightarrow R[X]$ is an $R[X]$-module homomorphism. For each $j\in \N$, define $\psi_j:P\rightarrow R$ by taking $\psi_j(q)$ to be the coefficient of $X^j$ in $\psi(q)$. Then each $\psi_j$ is an $R$-module homomorphism and for each $q\in P$, we have $\psi_j(q)=0$ for all but finitely many $j$. For $q(X)=\sum_i q_i X^i$, we have 
	\begin{equation*}
		\psi(q(X))=\sum_i\psi(q_i)X^i=\sum_i\bigg(\sum_j\psi_j(q_i)X^j\bigg)X^i=\sum_{i,j}\psi_j(q_i)X^{i+j}.
	\end{equation*}
	In addition, if $P$ is finitely generated, then there exists $n_0\in\N$ such that $\psi_j\equiv 0$ for all $j> n_0$. (If $P$ is generated by $\{p_1,p_2,\ldots,p_n\}$, then take $n_0=\mathrm{max}\{\mathrm{deg}(\psi(p_i))\}$.) 
\end{remark}
	
\begin{lemma}
\label{lemma:P[X] is symplectic}
	Let $(P,\langle,\rangle)$ be a symplectic $R$-module with $P$ finitely generated projective module of even rank. Define $\langle,\rangle_{\otimes R[X]}:P[X]\times P[X]\rightarrow R[X]$ as
	\begin{equation*}
		\bigg\langle\sum_ip_iX^i,\sum_jq_jX^j\bigg\rangle_{\otimes R[X]}=\sum_{i,j}\langle p,q\rangle X^{i+j}.
	\end{equation*}
	Then, $\langle,\rangle_{\otimes R[X]}$ is a nondegenerate alternating bilinear form on $P[X]$.
\end{lemma}
\begin{proof}
	Note that $\langle,\rangle_{\otimes R[X]}$ is a bilinear form as $\langle,\rangle$ is bilinear. For $p(X)=\sum p_iX^i\in P[X]$, we have 
	\begin{equation*}
		\bigg\langle p(X), p(X)\bigg\rangle_{\otimes R[X]}=\sum_j\bigg(\sum_i \langle p_i,p_{j-i}\rangle X^j\bigg).
	\end{equation*}
	The bilinear form $\langle,\rangle$ is alternating and hence is skew-symmetric. This implies that $\langle p_i,p_{j-i}\rangle+\langle p_{j-i},p_i\rangle=0$ for all $i$, $j$. Therefore, $\langle p(X), p(X)\rangle_{\otimes R[X]}=0$, for all $p(X)\in P[X]$ and hence $\langle,\rangle_{\otimes R[X]}$ is alternating.
		
	The map $\varphi:P\rightarrow P^\ast$ given by $\varphi(p)=\langle p,-\rangle$ is an isomorphism as $\langle,\rangle$ is nondegenerate. Define $\varphi:P[X]\rightarrow P[X]^\ast$ as $\varphi(p(X))=\langle p(X),- \rangle_{\otimes R[X]}$. Then $\varphi$ is an $R[X]$-module homomorphism as $\langle,\rangle_{\otimes R[X]}$ is bilinear. Suppose $p(X)=\sum p_iX^i\in P[X]$ is such that $\varphi(p(X))=0$. Then, we have $\langle p(X), q \rangle_{\otimes R[X]} =0$ for all $q\in P$, that is, $\sum\langle p_i, q\rangle X^i=0$ for all $q\in P$. This means that $\varphi(p_i)=0$ for all $i$. As $\varphi$ is injective, we have $p_i=0$ for all $i$. Therefore, $p(X)=0$ and hence $\varphi$ is injective.
		
	Let $\psi:P[X]\rightarrow R[X]$ be an $R[X]$-module homomorphism. For $j\in\N$, let $\psi_j$s be defined as in remark \ref{remark:surjetive}. Then we have 
	\begin{equation*} 
		\psi\bigg(\sum_i q_iX^i\bigg)=\sum_{i,j}\psi_j(q_i)X^{i+j}.
	\end{equation*}
	As $\varphi$ is surjective and $\psi_j\in P^\ast$ for $i=1,2,\ldots,n_0$, we have $\psi_j=\varphi(p_j)=\langle p_j,-\rangle$ for some $p_j\in P$. Then, 
	\begin{equation*} 
		\psi\bigg(\sum_i q_iX^i\bigg)=\sum_{i,j}\langle p_j,q_i\rangle X^{i+j}
			=\bigg\langle \sum_jp_jX^j,\sum_i q_iX^i\bigg\rangle_{\otimes R[X]}.
	\end{equation*} 
	Therefore, for every $\psi\in P[X]^\ast$, there exists $p(X)\in P[X]$ such that $\psi=\langle p(X),-\rangle_{\otimes R[X]}=\varphi(p(X))$ and hence $\varphi$ is surjective. This proves that $\langle,\rangle_{\otimes R[X]}$ is nondegenerate.
\qed          
\end{proof}
	
\medskip
	
\begin{remark}
	By Lemma \ref{lemma:P[X] is fin gen} and Lemma \ref{lemma:P[X] is proj}, we have $P[X]$ is a finitely generated projective module. The map $\langle,\rangle_{\otimes R[X]}$ defined in Lemma \ref{lemma:P[X] is symplectic} is a nondegenerate alternating bilinear form on $P[X]$. So, $(P[X],\langle,\rangle_{\otimes R[X]})$ is a symplectic $R[X]$-module.
\end{remark}


\section{$V(P,\langle,\rangle)$ and $V(P,IP,\langle,\rangle)$}
\label{section:5}
	
In this section, we define a group which is considered as generalization of $\mathrm{ESp}_\varphi(R)$ in the case of projective modules.
	
\medskip
\noindent
\textbf{Notation:} Suppose $P$ is a free $R$-module and $\varphi$ is an invertible alternating matrix of size $2n$ over $R$. We say an $R$-module homomorphism $\alpha:P\rightarrow P$ is an element of $\mathrm{ESp}_\varphi(R)$ if the matrix of $\alpha$, in some basis of $P$, is an element of $\mathrm{ESp}_\varphi(R)$. 
	
\begin{lemma}
\label{lemma:local basis independence}
	Let $(R,\ideal{m})$ be a local ring and $P$ be a free $R$-module of rank $2n$. Suppose $\varphi_1$ and $\varphi_2$ are two alternating matrices of Pfaffian $1$ of size $2n$ over $R$. Let $\alpha:P\rightarrow P$ be an $R$-module homomorphism. Then, $\alpha\in\mathrm{ESp}_{\varphi_1}(R)\text{ if and only if } \alpha\in\mathrm{ESp}_{\varphi_2}(R)$.
\end{lemma}
\begin{proof}
	As $\varphi_1$ and $\varphi_2$ are two alternating matrices of Pfaffian $1$ over $R$, we have $\varphi_1=(1\perp\epsilon_1)^t\psi_n(1\perp\epsilon_1)$ and $\varphi_2=(1\perp\epsilon_2)^t\psi_n(1\perp\epsilon_2)$ for some $\epsilon_1,\epsilon_2\in\mathrm{E}_{2n-1}(R)$, by Lemma \ref{lemma:3.2} and Lemma \ref{lemma:3.4}. Then, $\varphi_1=(1\perp\epsilon)^t\varphi_2(1\perp\epsilon)$, where $\epsilon=\epsilon_2^{-1}\epsilon_1$. By Lemma \ref{lemma:3.5}, we have $\mathrm{ESp}_{\varphi_1}(R)=(1\perp\epsilon)^{-1}\mathrm{ESp}_{\varphi_2}(R)(1\perp\epsilon)$. Let $\mathcal{B}_1$ be a basis of $P$ over $R$. Construct a basis $\mathcal{B}_2$ of $P$ such that the change of basis matrix from $\mathcal{B}_1$ to $\mathcal{B}_2$ is $(1\perp\epsilon)$.
		
	Consider an R-module homomorphism $\alpha:P\rightarrow P$ and let $[\alpha]_1$ and $[\alpha]_2$ denote the matrices of $\alpha$ in the bases $\mathcal{B}_1$ and $\mathcal{B}_2$ respectively. Then, $[\alpha]_1=(1\perp\epsilon)^{-1}[\alpha]_2(1\perp\epsilon)$. So,
	\begin{eqnarray*}
		&&[\alpha]_1\in\mathrm{ESp}_{\varphi_1}(R)\\
		&\iff &(1\perp\epsilon)^{-1}[\alpha]_2(1\perp\epsilon) \in\mathrm{ESp}_{\varphi_1}(R)\\
		&\iff &[\alpha]_2\in(1\perp\epsilon)\mathrm{ESp}_{\varphi_1}(R)(1\perp\epsilon)^{-1} \\
		&\iff &[\alpha]_2\in\mathrm{ESp}_{\varphi_2}(R).
	\end{eqnarray*}
	So, $\alpha\in\mathrm{ESp}_{\varphi_1}(R)\text{ if and only if } \alpha\in\mathrm{ESp}_{\varphi_2}(R)$.
\qed   
\end{proof}
	
\medskip
	
The following lemma (Lemma \ref{lemma:local basis independence relative}) gives a relative version of Lemma \ref{lemma:local basis independence} and can be proved in a similar way as Lemma \ref{lemma:local basis independence} using Lemma \ref{lemma:3.6}.
	
\begin{lemma}
\label{lemma:local basis independence relative}
	Let $(R,\ideal{m})$ be a local ring and $P$ be a free $R$-module of rank $2n$. Let $I$ be an ideal of $R$. Suppose $\varphi_1$ and $\varphi_2$ are two alternating matrices of Pfaffian $1$ of size $2n$ over $R$. Let $\alpha:P\rightarrow P$ be an $R$-module homomorphism. Then, $\alpha\in\mathrm{ESp}_{\varphi_1}(R,I)\text{ if and only if } \alpha\in\mathrm{ESp}_{\varphi_2}(R,I)$.
\end{lemma}
	
\medskip
	
\noindent
\textbf{Notation:} Consider an $R[X]-$module homomorphism $\alpha[X]: P[X]\rightarrow P[X]$. For $a\in R$, define $f_a:P[X]\rightarrow P$ as $f_a(p(X))=p(a)$. Then, we denote by $\alpha(a)$ the map $\alpha(a): P\rightarrow P$ given by $\alpha(a):=f_a\circ\alpha(X)|_P$. 
	
\begin{definition} 
	\label{V(P) defn}
	\rm{Let $(P,\langle,\rangle)$ be a symplectic $R$-module with $P$ finitely generated projective module of even rank. We define $V(P,\langle,\rangle)$ to be the collection
	\begin{equation*}
		\begin{aligned}
		\{\alpha(1)\;\big|\;\alpha(X)\in\mathrm{Sp}(P[X],\langle,\rangle_{\otimes R[X]}),\; \alpha(0)=Id., \text{ and } \\
		\alpha(X)_\ideal{p}\in\mathrm{ESp}_{\varphi\otimes R_\ideal{p}[X]}(R_\ideal{p}[X]), \text{ for all }\ideal{p}\in Spec(R)\}
		\end{aligned}
	\end{equation*}
	where $\langle,\rangle$ corresponds to an alternating matrix $\varphi$ (with respect to some basis) of Pfaffian $1$ over the local ring $R_\ideal{p}$ at the prime ideal $\ideal{p}$. This definition is independent of the choice of local basis in view of Lemma \ref{lemma:local basis independence}.}
\end{definition} 
	
\begin{definition} 
	\label{V(P,IP) defn}
	\rm{Let $(P,\langle,\rangle)$ be a symplectic $R$-module with $P$ finitely generated projective module of even rank. Let $I$ be an ideal of $R$. We define $V(P,IP,\langle,\rangle)$ to be the collection
	\begin{equation*}
		\begin{aligned}
		\{\alpha(1)\;\big|\;\alpha(X)\in\mathrm{Sp}(P[X],\langle,\rangle_{\otimes R[X]}),\; \alpha(0)=Id., \text{ and } \\
		\alpha(X)_\ideal{p}\in\mathrm{ESp}_{\varphi\otimes R_\ideal{p}[X]}(R_\ideal{p}[X],I_\ideal{p}[X]), \text{ for all }\ideal{p}\in Spec(R)\}
		\end{aligned}
	\end{equation*}
	where $\langle,\rangle$ corresponds to an alternating matrix $\varphi$ (with respect to some basis) of Pfaffian $1$ over the local ring $R_\ideal{p}$ at the prime ideal $\ideal{p}$. This definition is independent of the choice of local basis in view of Lemma \ref{lemma:local basis independence relative}.}
\end{definition} 
	
\begin{remark}
	If $I=R$, then $V(P,IP,\langle,\rangle)=V(P,\langle,\rangle)$ as $\mathrm{ESp}_\varphi(R,R)=\mathrm{ESp}_\varphi(R)$ for an invertible skew-symmetric matrix $\varphi$ over $R$. So, any results that hold for $V(P,IP,\langle,\rangle)$ will also hold for $V(P,\langle,\rangle)$.
\end{remark}
	
\medskip

\noindent
\textbf{Notation:} Let $p(X), q(X)\in P[X]$ and $a\in R$. Let $r(X):=\langle p(X), q(X)\rangle_{\otimes R[X]}\in R[X]$. Then $\langle p(X), q(X)\rangle_{\otimes R[X]}(a):=r(a)$.
	
\begin{remark}
	Let $p(X),q(X)\in P[X]$ and $a\in R$. Then, by definition and bilinearity of $\langle,\rangle_{\otimes R[X]}$, we have 			$\langle p(X),q(X)\rangle_{\otimes R[X]}(a)=\langle p(a),q(a)\rangle$. 
	If $\alpha(X):P[X]\rightarrow P[X]$ is an $R[X]$ module homomorphism, and $p,q\in P$, then 
	\begin{equation*}
		\langle \alpha(a)(p),\alpha(a)(q)\rangle=\langle \alpha(X)(p),\alpha(X)(q)\rangle_{\otimes R[X]}(a).
	\end{equation*}
\end{remark}
	
\begin{lemma}
\label{lemma:V is subgp of Sp}
	Let $(P,\langle,\rangle)$ be a symplectic $R$-module with $P$ finitely generated $R$-module of rank $2n$. Then, $V(P,\langle,\rangle)$ is a subgroup of the group of isometries of $P$.
\end{lemma}
\begin{proof}
	Let $\gamma\in V(P,\langle,\rangle)$. Then, by definition, there exists an isometry $\alpha(X)$ of $P[X]$ with $\alpha(X)_\ideal{p} \in \mathrm{ESp}_{\varphi\otimes R_\ideal{p}[X]}(R_\ideal{p}[X])$ for all $\ideal{p}\in Spec(R)$ such that $\alpha(0)=Id$ and $\alpha(1) = \gamma$. As $\alpha(X)_\ideal{p} \in \mathrm{ESp}_{\varphi\otimes R_\ideal{p}[X]}(R_\ideal{p}[X])$ for all $\ideal{p}\in Spec(R)$, we have $\alpha(1)_\ideal{p} \in \mathrm{ESp}_{\varphi\otimes R_\ideal{p}}(R_\ideal{p})$ for all $\ideal{p}\in Spec(R)$. In particular, $\alpha(1)_\ideal{p}$ is an isomorphism for all $\ideal{p}\in Spec(R)$. Hence, $\alpha(1):P\rightarrow P$ is an isomorphism. Also, for $p,q\in P$,
	\begin{eqnarray*}
		\langle\alpha(1)(p),\alpha(1)(q)\rangle
		&=&\langle\alpha(X)(p),\alpha(X)(q)\rangle_{\otimes R[X]}(1)\\
		&=&\langle p,q\rangle_{\otimes R[X]}(1) \text{ [as $\alpha(X)$ is an isometry]}\\
		&=&\langle p,q\rangle 
	\end{eqnarray*}
	Therefore $\gamma=\alpha(1)$ is an isometry of $P$ and hence $\gamma\in \mathrm{Sp}(P,\langle,\rangle)$. This gives us that $V(P,\langle,\rangle)$ is a subgroup of the group of isometries of $P$.
\qed        	    	      
\end{proof}

\medskip
Next, we will state a relative version of the above lemma which can be proved using a similar argument as above.

\begin{lemma}
    Let $(P,\langle,\rangle)$ be a symplectic $R$-module with $P$ finitely generated $R$-module of rank $2n$ and $I$ be an ideal of $R$. Then, $V(P,IP,\langle,\rangle)$ is a subgroup of the group of isometries of $P$.
\end{lemma}


\section{$V(P,IP,\langle,\rangle)$ in free case}
\label{section:7}
	
In this section, we prove the Local-Global principle for relative symplectic case, using which we prove that when $P$ is free, the group $V(P,IP,\langle,\rangle)$ equals the relative elementary symplectic group with respect to some alternating matrix. 

\begin{definition} 
	\rm{Let $R$ be a commutative ring and $I$ be an ideal of $R$.  The subgroup of $\mathrm{ESp}_{2n}(R,I)$ generated by the elementary symplectic matrices $se_{1j}(a)$, where $a\in R$, and $se_{i1}(x)$, where $x\in I$ is denoted by $\mathrm{ESp}_{2n}^1(R,I)$.}
\end{definition} 

\begin{lemma}(Lemma 2.10, \cite{JOA})
\label{ESp=Esp1 cap Sp}
    Let $R$ be a ring with $R=2R$ and $I$ be an ideal of $R$. Let $n\geq 2$. Then, $\mathrm{ESp}_{2n}(R,I)=\mathrm{ESp}_{2n}^1(R,I)\cap\mathrm{Sp}_{2n}(R,I)$
\end{lemma}	

\begin{remark}
	\label{f(0)+Xg(X)}
	If $f(X)\in R[X]$, then we can write $f(X)=f(0)+Xg(X)$ for some $g(X)\in R[X]$ and we have $se_{ij}(f(X))=se_{ij}(f(0))se_{ij}(Xg(X))$. If $f(X)\in I[X]$, then $f(0)\in I$ and $g(X)\in I[X]$.
\end{remark}

\begin{lemma}(Lemma 3.1, \cite{JPAA})
\label{Y^rX}
    Let $R$ be a commutative ring with $R=2R$ and $I$ be an ideal of $R$. Suppose $n\geq 2$. If $\gamma=\gamma_1\cdots\gamma_r\in\mathrm{ESp}_{2n}^1(R,I)$, where each $\gamma_i$ is an elementary symplectic matrix, and $se_{ij}(Xf(X))$ is an elementary generator of $\mathrm{ESp}_{2n}^1(R[X],I[X])$, then
    \begin{equation*}
    	\gamma se_{ij}(Y^{4^r}Xf(Y^{4^r}X))\gamma^{-1}=\prod se_{i_tj_t}(Yh_t(X,Y)), 
    \end{equation*}
    where $i_t=1$ or $j_t=1$ and $h_t(X,Y)\in\begin{cases}
    	R[X,Y]\text{ if }i_t=1\\I[X,Y]\text{ if }j_t=1. \end{cases}$
\end{lemma}

\begin{lemma}
\label{prod aibi}
	Let $G$ be a group and $a_i,b_i\in G$ for $i=1,\ldots,n$. Then 
	\begin{equation*}
		\prod_{i=1}^na_ib_i=\bigg( \prod_{i=1}^n \big(\prod_{j=1}^i a_j \big) b_i \big(\prod_{j=1}^i a_j \big)^{-1} \bigg)\prod_{i=1}^n a_i.
	\end{equation*}
\end{lemma}

\begin{remark}
	\label{ESp psi n=ESp 2n}
	Let $v^t=(a_1,\ldots,a_{2n-1})\in R^{2n-1}$. Then,
	\begin{equation*}
		C_{\psi_n}(v)=\prod_{i=2}^{2n-1}se_{i1}(a_{i-1})\text{ and }R_{\psi_n}(v)=\prod_{j=2}^{2n-1}se_{1j}(a_{i-1})
	\end{equation*}
\end{remark}

\begin{remark}
	\label{ESp phi Esp phi*}
	If $\varphi$ and $\varphi^\ast$ are invertible alternating matrices of size $2n$ over $R$ such that $\varphi=(1\perp\epsilon)^t\varphi^\ast(1\perp\epsilon)$ for some $\epsilon\in \mathrm{E}_{2n-1}(R)$, then 
	\begin{eqnarray*}
		(1\perp\epsilon)^{-1}C_{\varphi^\ast}(v)(1\perp\epsilon)=C_\varphi(\epsilon^{-1}v),\\ (1\perp\epsilon)^{-1}R_{\varphi^\ast}(v)(1\perp\epsilon)=R_\varphi(\epsilon^tv).
	\end{eqnarray*} 
\end{remark}

\begin{lemma}(Dilation principle)
\label{lemma:Dilation}
	Let $R$ be a ring with $R=2R$ and $I$ be an ideal of $R$. Let $a\in R$ be non-nilpotent. Let $\varphi$ be an invertible alternating matrix of size $2n$ such that $\varphi=(1\perp\epsilon)^t\psi_n(1\perp\epsilon)$ for some $\epsilon\in\mathrm{E}_{2n-1}(R_a)$. Suppose $\alpha(X)\in\mathrm{ESp}_{\varphi\otimes R_a[X]}(R_a[X],I_a[X])$ with $\alpha(0)=I_{2n}$. Then, there exists $\alpha^\ast(X)\in\mathrm{ESp}_{\varphi\otimes R[X]}(R[X],I[X])$ with $\alpha^\ast(0)=I_{2n}$ such that $\alpha^\ast(X)$ localises to $\alpha(bX)$ for some $b\in (a^d)$, $d\gg 0$.  
\end{lemma}
\begin{proof}
	Over $R_a$, we have $\varphi=(1\perp\epsilon)^t\psi_n(1\perp\epsilon)$ for some $\epsilon\in\mathrm{E}_{2n-1}(R_a)$ and hence
	\begin{equation*}
		\mathrm{ESp}_{\varphi\otimes R_a[X]}(R_a[X],I_a[X])=(1\perp\epsilon)^{-1}\mathrm{ESp}_{2n}(R_a[X],I_a[X])(1\perp\epsilon).
	\end{equation*} 
	As $\alpha(X)\in\mathrm{ESp}_{\varphi\otimes R_a[X]}(R_a[X],I_a[X])$, there exists $\beta(X)\in \mathrm{ESp}_{2n}(R_a[X],I_a[X])$ such that $\alpha(X)=(1\perp\epsilon)^{-1}\beta(X)(1\perp\epsilon)$ and $\beta(0)=I_{2n}$. By Lemma \ref{ESp=Esp1 cap Sp}, we have
	\begin{equation*}
		\beta(X)=\prod se_{i_k,j_k}(f_k(X)),\text{ where }i_k=1\text{ or }j_k=1\text{ and }f_k\in\begin{cases}
			R_a[X]\text{ if }i_k=1\\I_a[X]\text{ if }j_k=1.
		\end{cases}
	\end{equation*}
    By remark \ref{f(0)+Xg(X)}, 
    \begin{eqnarray*}
    	\beta(X)=\prod se_{i_k,j_k}(f_k(0))se_{i_k,j_k}(Xg_k(X)),\text{ where }i_k=1\text{ or }j_k=1\text{ and }\\f_k(0)\in\begin{cases}
    		R_a\text{ if }i_k=1\\I_a\text{ if }j_k=1\end{cases},\;g_k(X)\in\begin{cases}
    			R_a[X]\text{ if }i_k=1\\I_a[X]\text{ if }j_k=1.
    		\end{cases}.
    \end{eqnarray*}
    As $\beta(0)=I_{2n}$, we have $\prod se_{i_kj_k}(f_k(0))=I_{2n}$. Hence, by Lemma \ref{prod aibi}, $\beta(X)=\prod\gamma_k se_{i_kj_k}(Xg_k(X))\gamma_k^{-1}$. Now, by Lemma \ref{Y^rX}, for $r\gg 0$
    \begin{eqnarray*}
    	\beta(Y^rX)=\prod se_{i_t,j_t}\bigg(\frac{Yh_t(X,Y)}{a^s}\bigg),\text{ where }i_t=1\text{ or }j_t=1\text{ and }\\h_t(X,Y)\in
    		R[X,Y]\text{ if }i_t=1,\;h_t(X,Y)\in I[X,Y]\text{ if }j_t=1.
    \end{eqnarray*}
    By remark \ref{ESp psi n=ESp 2n}, 
    \begin{eqnarray*}
    	\beta(Y^rX)=\prod\delta_t\bigg(\frac{Yh_t(X,Y)}{a^s}e_{k_t}\bigg),\text{ where }\\
    	\delta_t:=\begin{cases}
    		C_{\psi_n}\text{ if }j_t=1\\R_{\psi_n}\text{ if }i_t=1
    	\end{cases}\text{and }k_t=\begin{cases}
    	i_t\text{ if }j_t=1\\j_t\text{ if }i_t=1.
    \end{cases}
    \end{eqnarray*}
    Define $v_t(X,Y):=\begin{cases}
    \epsilon^{-1}\frac{Yh_t(X,Y)}{a^s}e_{k_t}\text{ if } \delta_t= C_{\psi_n}\\
    \epsilon^t\frac{Yh_t(X,Y)}{a^s}e_{k_t}\text{ if } \delta_t= R_{\psi_n}
    \end{cases}$. Then, by remark \ref{ESp phi Esp phi*}, 
    \begin{equation*}
    	\alpha(Y^rX)=\prod \gamma_t(v_t(X,Y))\text{ where } \gamma_t:=\begin{cases}C_\varphi\text{ if }\delta_t= C_{\psi_n}\\R_\varphi\text{ if }\delta_t= R_{\psi_n}.\end{cases}
    \end{equation*}
    Let $l$ be the maximum power of $a$ in the denominator in $\epsilon^{-1}$ and $\epsilon^t$ and take $d=s+t$. Define $\alpha^\ast(X,Y)=\prod \gamma_t(v_t(X,a^dY))$. Then, \begin{equation*}\alpha^\ast(X,Y)\in\mathrm{ESp}_{\varphi\otimes R[X,Y]}(R[X,Y],I[X,Y]).\end{equation*} Take $\alpha^\ast(X)=\alpha^\ast(X,1)$. Then, $\alpha^\ast(X)\in\mathrm{ESp}_{\varphi\otimes R[X]}(R[X],I[X])$ with $\alpha^\ast(0)=I_{2n}$ and $\alpha^\ast(X)$ localises to $\alpha(a^{rd}X)$.
\qed
\end{proof}
    
\begin{remark}(Lemma 5.2, \cite{JPAA})
\label{phi over local ring}
	Suppose $(R,\ideal{m})$ is a local ring and $\varphi$ is an invertible alternating matrix of size $2n$ of Pfaffian $1$ over $R$. Then, there exists $\epsilon\in\mathrm{E}_{2n-1}(R)$ such that $\varphi=(1\perp\epsilon)^t\psi_n(1\perp\epsilon)$. 
\end{remark}	

\begin{lemma}{(Local-Global principle)}
\label{lemma:LG principle}
	Let $R$ be a ring with $R=2R$. Let $\varphi$ be a skew-symmetric matrix of Pfaffian $1$ of size $2n$ over $R$ with $n\geq 2$. Let $\theta(X)\in\mathrm{Sp}_{\varphi\otimes R[X]}(R[X])$, with $\theta(0)=I_{2n}$. If $\theta(X)_\ideal{m}\in\mathrm{ESp}_{\varphi\otimes R_\ideal{m}[X]}(R_\ideal{m}[X],I_\ideal{m}[X])$, for all maximal ideals $\ideal{m}$ of $R$, then $\theta(X)\in\mathrm{ESp}_{\varphi\otimes R[X]}(R[X],I[X])$. 
\end{lemma}
\begin{proof}
    Let $\ideal{m}$ be a maximal ideal. By remark \ref{phi over local ring}, there exists $\epsilon\in\mathrm{E}_{2n-1}(R_\ideal{m})$ such that $\varphi=(1\perp\epsilon)^t\psi_n(1\perp\epsilon)$. As $\epsilon\in\mathrm{E}_{2n-1}(R_\ideal{m})$ and $\theta(X)_\ideal{m}\in\mathrm{ESp}_{\varphi\otimes R_\ideal{m}[X]}(R_\ideal{m}[X],I_\ideal{m}[X])$, there exists $a_\ideal{m}\in R\setminus\ideal{m}$ such that 
    \begin{equation*}
    	\epsilon\in\mathrm{E}_{2n-1}(R_{a_\ideal{m}}) \text{ and }  \theta(X)_{a_\ideal{m}}\in\mathrm{ESp}_{\varphi\otimes R_{a_\ideal{m}}[X]}(R_{a_\ideal{m}}[X],I_{a_\ideal{m}}[X]).
    \end{equation*} (Choose $a_\ideal{m}$ to be the product of all the denominators in $\epsilon$ and $\theta(X)_\ideal{m}$).
    
    Define $\gamma(X,Y):=\theta(X+Y)_{a_\ideal{m}}\theta(Y)_{a_\ideal{m}}^{-1}$. Then, $\gamma(0,Y)=I_{2n}$ and 
    \begin{equation*}
    	\gamma(X,Y)\in\mathrm{ESp}_{\varphi\otimes R_{a_\ideal{m}}[X,Y]}(R_{a_\ideal{m}}[X,Y],I_{a_\ideal{m}}[X,Y]).
    \end{equation*} 
    By Lemma \ref{lemma:Dilation}, $\gamma(b_\ideal{m}X,Y)\in\mathrm{ESp}_{\varphi\otimes R[X,Y]}(R[X,Y],I[X,Y])$ for some $b_\ideal{m}\in(a_\ideal{m}^d)$.
    
    For every maximal ideal $\ideal{m}$, we have $a_\ideal{m}^d\in R\setminus\ideal{m}$. Therefore, the set $\{a_\ideal{m}^d\; :\; \ideal{m}\text{ is a maximal ideal of }R\}$ generates $R$. There exist $\ideal{m}_1,\ldots,\ideal{m}_k$, maximal ideals of $R$ and $c_1,\ldots,c_k\in R$ such that $c_1a_{\ideal{m}_1}^d+\cdots+c_ka_{\ideal{m}_k}^d=1$. Let $b_i:=c_ka_{\ideal{m}_k}^d$. Then,
    \begin{equation*}
    	\theta(X)=\prod_{i=1}^k\gamma(b_iX,T_i)\text{ where } T_i=\sum_{j=i+1}^kb_jX\text{ for }1\leq i\leq k-1\text{ and }T_k=0.
    \end{equation*}
    As $\gamma(b_iX,T_i)\in\mathrm{ESp}_{\varphi\otimes R[X]}(R[X],I[X])$ for $1\leq i\leq k$, we have 
    \begin{equation*}
    	\theta(X)\in\mathrm{ESp}_{\varphi\otimes R[X]}(R[X],I[X]).
    \end{equation*}
\qed
\end{proof}
	
\begin{lemma}
\label{lemma:V=ESp}
	Let $R$ be a ring with $R=2R$. Let $(P,\langle,\rangle)$ be a symplectic $R$-module, where $P$ is free of rank $2n$. Suppose $\langle,\rangle$ corresponds to an alternating matrix $\varphi$ of size $2n$ over $R$. Then, 
	\begin{equation*}
		V(P,IP,\langle,\rangle_\varphi)=\mathrm{ESp}_\varphi(R,I)
	\end{equation*}
\end{lemma}
\begin{proof}
	Any element of $\mathrm{ESp}_\varphi(R,I)$ is a product of elements of the form $\alpha=\gamma(u)^{-1}\delta(v)\gamma(u)$, where $u\in R$, $v\in I$ and $\gamma,\delta$ denote $C_\varphi$ or $R_\varphi$. Define
	\begin{equation*}
		\alpha(X)=\gamma(uX)^{-1}\delta(vX)\gamma(uX)\in \mathrm{ESp}_{\varphi\otimes R[X]}(R[X],I[X]).
	\end{equation*}  
	Then $\alpha(0)=Id$ and  $\alpha(1)=\alpha$. For a prime ideal $\ideal{p}$ of $R$, we have 
	\begin{equation*}
		\alpha(X)_\ideal{p}=\gamma(u_\ideal{p}X)^{-1}\delta(v_\ideal{p}X)\gamma(u_\ideal{p}X)\in\mathrm{ESp}_{\varphi\otimes R_\ideal{p}[X]}(R_\ideal{p}[X],I_\ideal{p}[X])
	\end{equation*} 
	Then, by definition, we have $\alpha=\alpha(1)\in V(P,IP,\langle,\rangle_\varphi)$. Therefore, $\mathrm{ESp}_\varphi(R,I)\subseteq V(P,IP,\langle,\rangle_\varphi)$.
		
	Let $\gamma\in V(P,IP,\langle,\rangle_\varphi)$. Then, by definition, there exists $\alpha(X)\in \mathrm{Sp}_{\varphi\otimes R[X]}(R[X])$ with $\alpha(X)_\ideal{p} \in \mathrm{ESp}_{\varphi\otimes R_\ideal{p}[X]}(R_\ideal{p}[X],I_\ideal{p}[X])$ for all $\ideal{p}\in Spec(R)$ such that $\alpha(0)=I_{2n}$ and $\alpha(1) = \gamma$. By Lemma \ref{lemma:LG principle}, $\alpha(X)\in \mathrm{ESp}_{\varphi\otimes R[X]}(R[X],I[X])$. Substituting $X=1$, we get $\gamma=\alpha(1)\in\mathrm{ESp}_\varphi(R,I)$. Therefore, the groups $V(P,IP,\langle,\rangle_\varphi)$ and $\mathrm{ESp}_\varphi(R,I)$ are equal. 
\qed
\end{proof}
	
\begin{remark}
	In view of Lemma \ref{lemma:V=ESp}, we can consider $V(P,\langle,\rangle)$ and $V(P,IP,\langle,\rangle)$ as generalizations of the elementary symplectic groups ($\mathrm{ESp}_\varphi(R)$ and $\mathrm{ESp}_\varphi(R,I)$ respectively) in the free case.
\end{remark}


\section{Normality of $V(P,IP,\langle,\rangle)$}
\label{section:6}
	
In this section, we prove the main result of this paper, namely normality of $V(P,IP,\langle,\rangle)$ in $\mathrm{Sp}(P,\langle,\rangle)$ (see Lemma \ref{lemma:normality}).
	
\begin{lemma}
\label{lemma:existence}
	Let $(P,\langle,\rangle)$ be a symplectic $R$-module with $P$ finitely generated $R$-module of rank $2n$. Suppose $\alpha$ is an isometry of $P$. Then there exists $\alpha(X)\in\mathrm{Sp}(P[X],\langle,\rangle_{\otimes R[X]})$ such that $\alpha(1)=\alpha$.
\end{lemma}
\begin{proof}
	Define $\alpha(X):P[X]\rightarrow P[X]$ as $\alpha(X)(\sum p_iX^i):=\sum \alpha(p_i)X^i$. Then, we have $\alpha(1)=\alpha$ and $\alpha(X)$ is an $R[X]$-module homomorphism as $\alpha$ is an $R$-module homomorphism. Also, if $\alpha(X)(\sum p_iX^i)=0$, then $\alpha(p_i)=0$ for all $i$, which means that $p_i=0$ for all $i$ (as $\alpha$ is injective). Hence $\alpha(X)$ is injective. The homomorphism $\alpha(X)$ is also surjective as $\alpha$ is surjective. For $p(X)=\sum p_iX^i$, $q(X)=\sum q_jX^j\in P[X]$, we have 
	\begin{eqnarray*}
		\langle\alpha(X)(p(X)),\alpha(X)(q(X))\rangle_{\otimes R[X]}
		&=&\bigg\langle\sum\alpha(p_i)X^i,\sum\alpha(q_j)X^j\bigg\rangle_{\otimes R[X]}\\
		&=&\sum_{i,j}\langle\alpha(p_i),\alpha(q_j)\rangle X^{i+j}\\
		&=&\sum_{i,j}\langle p_i,q_j\rangle X^{i+j} \hfill \text{ [as $\alpha$ is an isometry of $P$]}\\
		&=&\langle p(X),q(X)\rangle_{\otimes R[X]}.
	\end{eqnarray*}
	Therefore $\alpha(X)$ is an isometry of $P[X]$.   
\qed
\end{proof}
	
\begin{lemma}
\label{lemma:normality}
	Let $(P,\langle,\rangle)$ be a symplectic $R$-module with $P$ finitely generated $R$-module of rank $2n$. Then, $V(P,IP,\langle,\rangle)$ is a normal subgroup of the group of isometries of $P$.
\end{lemma}
\begin{proof}
	Let $\gamma\in V(P,IP,\langle,\rangle)$. Then, by definition, there exists an isometry $\alpha(X)$ of $P[X]$ with $\alpha(X)_\ideal{p} \in \mathrm{ESp}_{\varphi\otimes R_\ideal{p}[X]}(R_\ideal{p}[X],I_\ideal{p}[X])$ for all $\ideal{p}\in Spec(R)$ such that $\alpha(0)=Id$ and $\alpha(1) = \gamma$. Suppose $\delta\in\mathrm{Sp}(P,\langle,\rangle)$. Then, by Lemma \ref{lemma:existence}, there exists $\delta(X)\in\mathrm{Sp}(P[X], \langle, \rangle_{\otimes R[X]})$ such that $\delta(1)=\delta$. Define $\beta(X)=\delta(X)^{-1}\alpha(X)\delta(X)$. Then $\beta(X)\in \mathrm{Sp}(P[X],\langle,\rangle_{\otimes R[X]})$ with $\beta(0)=Id$ and $\beta(1)=\delta^{-1}\gamma\delta$.
		
	Let $\ideal{p}$ be a prime ideal of $R$. Then $P_\ideal{p}[X]$ is $R_\ideal{p}[X]$ free of rank $2n$. Suppose $\langle,\rangle_{\otimes R[X]}$ corresponds to an invertible alternating matrix $\varphi$ of size $2n$ over $R[X]$. Then, by Lemma \ref{lemma:iso gp=sym gp} we have 
	\begin{equation}
	\label{eqn:iso gp=sym gp}
		\mathrm{Sp}(P_\ideal{p}[X],\langle,\rangle_{\otimes R[X]})= \mathrm{Sp}_{\varphi\otimes R_\ideal{p}[X]}(R_\ideal{p}[X]).
	\end{equation} 
	Now we have $\alpha(X)_\ideal{p}\in \mathrm{ESp}_{\varphi\otimes R_\ideal{p}[X]}(R_\ideal{p}[X],I_\ideal{p}[X])$ by choice of $\alpha(X)$ and $\delta(X)_\ideal{p}\in\mathrm{Sp}_{\varphi\otimes R_\ideal{p}[X]}(R_\ideal{p}[X])$ by equation \ref{eqn:iso gp=sym gp}. By Lemma \ref{lemma:norm of esp}, we have
	\begin{equation*}
		\beta(X)_\ideal{p}=(\delta(X)_\ideal{p})^{-1}(\alpha(X))_\ideal{p} (\delta(X))_\ideal{p} \in \mathrm{ESp}_{\varphi\otimes R_\ideal{p}[X]}(R_\ideal{p}[X],I_\ideal{p}[X]). 
	\end{equation*}
	Then, by definition $\beta(1)\in V(P,IP,\langle,\rangle)$, that is, $\delta^{-1}\gamma\delta\in V(P,IP,\langle,\rangle)$. Therefore, the group $V(P,IP,\langle,\rangle)$ is a normal subgroup of $\mathrm{Sp}(P,\langle,\rangle)$.  
\qed
\end{proof}


\end{document}